\documentclass[pdflatex,sn-mathphys-ay]{sn-jnl}

\usepackage{amsmath,amssymb,amsfonts}
\usepackage{amsthm}
\usepackage{mathtools}
\usepackage{booktabs}

\theoremstyle{thmstyleone}
\newtheorem{theorem}{Theorem}[section]
\newtheorem{proposition}[theorem]{Proposition}
\newtheorem{corollary}[theorem]{Corollary}

\theoremstyle{thmstyletwo}

\newtheorem{remark}[theorem]{Remark}

\theoremstyle{thmstylethree}
\newtheorem{definition}[theorem]{Definition}

\DeclareMathOperator{\Var}{Var}
\DeclareMathOperator{\Exp}{Exp}
\newcommand{\E}{\mathbb{E}}
\newcommand{\Pp}{\mathbb{P}}
\newcommand{\R}{\mathbb{R}}
\newcommand{\1}{\mathbf{1}}
\newcommand{\Law}{\mathcal{L}}
\newcommand{\Mclass}{\mathcal{M}}
\newcommand{\mgf}{\mathrm{mgf}}

\begin{document}

\title{Optimal Sub-Gamma Scales for Spectrally Positive Infinitely Divisible Laws}

\author*[1]{\fnm{Yichuan} \sur{Chen}}\email{798956520@qq.com}

\affil*[1]{\orgdiv{School of Art and Design}, \orgname{Zhejiang Sci-Tech University},
\orgaddress{\city{Hangzhou}, \postcode{310018}, \country{China}}}

\abstract{Let \(X\) be a centered spectrally positive infinitely divisible random variable with finite nonzero variance \(V\). We determine the smallest scale \(c\) in a right-sided sub-gamma bound whose quadratic proxy is fixed at \(V\). A probability transform of the normalized Kolmogorov canonical measure produces a nonnegative variable \(R_X\) for which \(\log \E e^{tX}=Vt^2M_{R_X}(t)/2\). The scale problem therefore reduces to comparing the MGF of \(R_X\) with that of an exponential law, leading to an exact one-dimensional variational formula. If \(X\) is non-Gaussian and has a positive exponential moment, the optimal scale is at least \(\kappa_3(X)/(3V)\), with equality precisely when \(R_X\) belongs to the M-class of Klar and Müller. Equality in the standard sub-gamma envelope occurs only for the Gaussian law or, in the non-Gaussian case, for centered compound Poisson laws with exponential jumps. For a general compound Poisson law, \(R_X\) is the second-order equilibrium distribution of the jump size. Exact calculations for gamma and bounded two-point jumps show that the optimum may occur at the origin, at the positive MGF boundary, or at an interior point.}

\keywords{infinitely divisible distribution, sub-gamma inequality, optimal scale, equilibrium distribution, M-class}

\pacs[MSC Classification]{60E07, 60E15, 60G51, 60K10}

\maketitle

\section{Introduction}\label{sec:intro}

Sub-gamma bounds describe a random variable through a quadratic variance proxy and a linear scale. For a centered random variable \(X\), a right-sided bound has the form
\begin{equation}\label{eq:intro-sg}
 \log \E e^{tX}
 \leq
 \frac{v t^2}{2(1-ct)},
 \qquad 0\leq t<1/c,
\end{equation}
with \(v>0\) and \(c\geq0\). Such bounds lead to Bernstein-type tail estimates and are standard in concentration theory \citep{BoucheronLugosiMassart2013}. The pair \((v,c)\) is not uniquely determined. In particular, increasing either parameter preserves feasibility, which can obscure how sharp a stated inequality is.

This paper fixes \(v\) at the true variance
\[
 V=\Var(X)
\]
and optimizes only the scale \(c\). The choice is forced locally: since
\[
 \log \E e^{tX}=\frac{Vt^2}{2}+o(t^2)
 \qquad (t\to0),
\]
no smaller quadratic proxy can hold in a neighborhood of the origin. Once \(v=V\), however, the smallest feasible \(c\) depends on the entire positive moment generating function (MGF), not only on a finite number of moments. Our aim is to identify that dependence for centered spectrally positive infinitely divisible laws.

Previous deviation inequalities for infinitely divisible distributions use the Lévy measure directly \citep{Houdre2002}. Related results cover norms of infinitely divisible vectors \citep{HoudreMarchalReynaudBouret2008} and compound Poisson concentration \citep{KontoyiannisMadiman2006}. Bias and Stein transforms provide further concentration bounds and probabilistic representations \citep{GoldsteinIslak2014,ArrasHoudre2019,BarmanUpadhyeVellaisamy2025}. In the present problem, the quadratic coefficient is fixed at \(V\), leaving the smallest compatible scale to be determined.

Higher-order equilibrium distributions are used in renewal, reliability, and actuarial theory \citep{NandaJainSingh1996,WillmotDrekicCai2005,NairPreeth2009}. \citet{KlarMuller2003} defined the M-class by comparing the MGF of a nonnegative law with that of an exponential law of the same mean; this MGF order is weaker than coefficientwise moment order. Optimal proxy problems have also been formulated variationally \citep{ArbelMarchalNguyen2020,BarretoMarchalArbel2026,LeskelaValimaa2026}. The canonical transform below places the present scale problem in this MGF order.

For a nonnegative variable \(R_X\) constructed from the canonical measure, we prove the exact factorization
\begin{equation}\label{eq:intro-factor}
 \log \E e^{tX}
 =
 \frac{Vt^2}{2}M_{R_X}(t).
\end{equation}
It follows that the optimal scale is
\[
 c_*(X)
 =
 \sup_{t>0}
 \frac{1-M_{R_X}(t)^{-1}}{t},
\]
where infinite MGF values are allowed. When \(X\) is non-Gaussian and has a positive exponential moment, expansion at the origin gives
\[
 c_*(X)\geq\frac{\kappa_3(X)}{3V},
\]
and equality holds precisely when \(R_X\) belongs to the M-class. Equality in the standard envelope also has a complete characterization: apart from the Gaussian case, the equality laws are centered compound Poisson variables with exponential jumps. For general compound Poisson laws, \(R_X\) is the second-order equilibrium distribution of the jump size. Gamma and bounded two-point jumps then illustrate control at the origin, at the positive MGF boundary, and at an interior argument.

The factorization is established in Section~\ref{sec:transform}. Section~\ref{sec:optimum} gives the variational formula and its local and boundary consequences. The M-class criterion and the equality laws appear in Section~\ref{sec:mclass}. Structural operations are treated in Section~\ref{sec:structure}, followed by the compound Poisson interpretation in Section~\ref{sec:cp}. Section~\ref{sec:examples} contains the exact examples, and Section~\ref{sec:discussion} discusses scope and limitations.

\section{The canonical remainder transform}\label{sec:transform}

\subsection{Setting}

Let \(X\) be centered, infinitely divisible, and spectrally positive, with
\[
 0<V=\Var(X)<\infty.
\]
Its Gaussian variance component is denoted by \(a\geq0\), and its Lévy measure \(\nu\) is supported on \((0,\infty)\). The centered Lévy--Khintchine representation can be written
\begin{equation}\label{eq:LK}
 K_X(t):=\log\E e^{tX}
 =
 \frac{a t^2}{2}
 +
 \int_{(0,\infty)}
 \bigl(e^{tx}-1-tx\bigr)\,\nu(dx)
\end{equation}
for every real \(t\) at which the right-hand side is finite. Finite variance gives
\begin{equation}\label{eq:variance}
 V=a+\int_{(0,\infty)}x^2\,\nu(dx).
\end{equation}
These formulas are standard consequences of the canonical representation of finite-variance infinitely divisible laws \citep{Sato1999,SteutelVanHarn2004}.

Introduce the finite Kolmogorov canonical measure
\begin{equation}\label{eq:H}
 H=a\delta_0+x^2\nu(dx)
\end{equation}
on \([0,\infty)\). Its total mass is \(V\). Equation~\eqref{eq:LK} becomes
\begin{equation}\label{eq:canonical}
 K_X(t)
 =
 \int_{[0,\infty)}g_t(x)\,H(dx),
 \qquad
 g_t(x)=
 \begin{cases}
 \dfrac{e^{tx}-1-tx}{x^2},&x>0,\\[6pt]
 \dfrac{t^2}{2},&x=0.
 \end{cases}
\end{equation}
The probabilistic interpretation below uses only the normalized measure \(H/V\). Related uses of biasing to interpret canonical or Lévy measures can be found in \citet{PakesSapatinasFosam1996} and \citet{ArrasHoudre2019}.

Let
\[
 Y_X\sim H/V,
\]
let \(B\) be independent of \(Y_X\) with density
\[
 f_B(b)=2(1-b)\1_{\{0<b<1\}},
\]
and define
\begin{equation}\label{eq:Rdef}
 R_X=BY_X.
\end{equation}
Thus \(B\sim\operatorname{Beta}(1,2)\), and \(R_X\) is nonnegative.

\subsection{Exact factorization}

\begin{theorem}[Canonical factorization]\label{thm:factor}
For every real \(t\) for which \(K_X(t)\) is finite,
\begin{equation}\label{eq:factor}
 K_X(t)=\frac{Vt^2}{2}M_{R_X}(t),
 \qquad
 M_{R_X}(t)=\E e^{tR_X}.
\end{equation}
For \(t>0\), either side of \eqref{eq:factor} is finite if and only if the other side is finite.
\end{theorem}

\begin{proof}
For \(z\in\R\), direct integration gives
\begin{equation}\label{eq:beta-int}
 \E e^{zB}
 =
 2\int_0^1(1-b)e^{zb}\,db
 =
 \frac{2(e^z-1-z)}{z^2},
\end{equation}
where the value at \(z=0\) is understood by continuity. Hence, for \(y\geq0\),
\[
 g_t(y)
 =
 \frac{t^2}{2}\E e^{tBy}.
\]
Integrating this identity with respect to \(Y_X\sim H/V\), and applying Tonelli when \(t>0\), yields
\[
 \int g_t(y)\,H(dy)
 =
 \frac{Vt^2}{2}\E e^{tBY_X}.
\]
Equation~\eqref{eq:canonical} proves the result. The same nonnegative-integrand argument shows the equivalence of finiteness for \(t>0\).
\end{proof}

The next formulas make the transform explicit at the Lévy-measure level.

\begin{proposition}[Law of the transform]\label{prop:Rlaw}
The law of \(R_X\) has an atom
\begin{equation}\label{eq:Ratom}
 \Pp(R_X=0)=\frac{a}{V}.
\end{equation}
On \((0,\infty)\), it has density
\begin{equation}\label{eq:Rdensity}
 q_X(r)=\frac{2}{V}\int_{(r,\infty)}(x-r)\,\nu(dx),
 \qquad r>0,
\end{equation}
and survival function
\begin{equation}\label{eq:Rsurvival}
 \Pp(R_X>r)
 =
 \frac{1}{V}\int_{(r,\infty)}(x-r)^2\,\nu(dx),
 \qquad r\geq0.
\end{equation}
\end{proposition}

\begin{proof}
Because \(B>0\) almost surely, \(R_X=0\) precisely when \(Y_X=0\), which proves \eqref{eq:Ratom}. Conditional on \(Y_X=x>0\), the variable \(xB\) has density
\[
 r\longmapsto \frac{2(x-r)}{x^2}\1_{\{0<r<x\}}.
\]
Mixing this density against \(x^2\nu(dx)/V\) gives \eqref{eq:Rdensity}. Finally, Tonelli's theorem yields
\begin{align*}
 \Pp(R_X>r)
 &=
 \int_r^\infty q_X(s)\,ds\\
 &=
 \frac{2}{V}\int_{(r,\infty)}
 \left\{\int_r^x(x-s)\,ds\right\}\nu(dx)\\
 &=
 \frac{1}{V}\int_{(r,\infty)}(x-r)^2\,\nu(dx).
\end{align*}
\end{proof}

\subsection{The left side}

The same factorization gives a one-line left-tail consequence.

\begin{corollary}[Variance-exact left sub-Gaussian bound]\label{cor:left}
For every \(s\geq0\),
\begin{equation}\label{eq:left-cgf}
 K_X(-s)\leq\frac{Vs^2}{2}.
\end{equation}
Consequently, for every \(x\geq0\),
\begin{equation}\label{eq:left-tail}
 \Pp(X\leq-x)\leq\exp\left(-\frac{x^2}{2V}\right).
\end{equation}
The quadratic proxy \(V\) cannot be decreased in any neighborhood of the origin.
\end{corollary}

\begin{proof}
Since \(R_X\geq0\), \(M_{R_X}(-s)\leq1\). Equation~\eqref{eq:factor} proves \eqref{eq:left-cgf}; Chernoff optimization over \(s\geq0\) gives \eqref{eq:left-tail}. Local minimality follows from \(K_X''(0)=V\).
\end{proof}

No scale parameter is needed for the left tail. On the right, the optimal scale depends on the growth of \(M_{R_X}(t)\), both near zero and up to its MGF boundary.

\section{The exact optimal scale}\label{sec:optimum}

\subsection{A variational formula}

For \(c\geq0\), write \(1/c=\infty\) when \(c=0\). We call \(c\) feasible if
\begin{equation}\label{eq:feasible}
 K_X(t)
 \leq
 \frac{Vt^2}{2(1-ct)}
 \qquad\text{for all }0\leq t<1/c.
\end{equation}
If \(K_X(t)=\infty\) at one of these arguments, the inequality fails. Define
\begin{equation}\label{eq:cstar}
 c_*(X)=\inf\{c\geq0:c\text{ is feasible}\},
\end{equation}
with the infimum of the empty set equal to \(\infty\).

By Theorem~\ref{thm:factor}, feasibility is equivalent to
\begin{equation}\label{eq:expdom}
 M_{R_X}(t)\leq\frac{1}{1-ct},
 \qquad 0\leq t<1/c.
\end{equation}
The right-hand side is the MGF of an exponential variable with mean \(c\).

For \(t>0\), set
\begin{equation}\label{eq:Psi}
 \Psi_X(t)
 =
 \frac{1-M_{R_X}(t)^{-1}}{t},
\end{equation}
using the convention \(1/\infty=0\).

\begin{theorem}[Variational characterization]\label{thm:variational}
For every centered spectrally positive infinitely divisible \(X\) with variance \(V\in(0,\infty)\),
\begin{equation}\label{eq:variational}
 \boxed{
 c_*(X)=\sup_{t>0}\Psi_X(t).
 }
\end{equation}
If the supremum is finite, then it is a feasible scale, so the infimum in \eqref{eq:cstar} is attained.
\end{theorem}

\begin{proof}
Suppose first that \(c\) is feasible. For \(0<t<1/c\), equation~\eqref{eq:expdom} rearranges to \(\Psi_X(t)\leq c\). If \(t\geq1/c\), then
\[
 0\leq\Psi_X(t)\leq\frac1t\leq c,
\]
because \(M_{R_X}(t)\geq1\), including the extended value \(+\infty\). Thus
\[
 \sup_{t>0}\Psi_X(t)\leq c.
\]
Taking the infimum over feasible \(c\) gives one inequality in \eqref{eq:variational}.

Conversely, let
\[
 C=\sup_{t>0}\Psi_X(t)<\infty.
\]
Fix \(0<t<1/C\), with the evident interpretation when \(C=0\). If \(M_{R_X}(t)=\infty\), then \(\Psi_X(t)=1/t>C\), a contradiction. Hence \(M_{R_X}(t)<\infty\), and
\[
 1-M_{R_X}(t)^{-1}\leq Ct.
\]
Because \(1-Ct>0\), this is equivalent to
\[
 M_{R_X}(t)\leq\frac1{1-Ct}.
\]
Thus \(C\) is feasible. This proves the reverse inequality and attainment.
\end{proof}

The variational scale gives the usual sub-gamma tail bound. If \(c_*(X)<\infty\), then for every \(u\geq0\),
\begin{equation}\label{eq:right-tail}
 \Pp\left(X\geq\sqrt{2Vu}+c_*(X)u\right)\leq e^{-u}.
\end{equation}
The remaining sections identify when this scale is controlled locally, at the MGF boundary, or at an interior point.

\subsection{Local cumulants and the MGF boundary}

Assume for the rest of this subsection that \(X\) has a nontrivial positive exponential moment. Then all positive moments of order at least two that occur below are finite, and \(K_X\) is analytic near zero. Comparing coefficients in \eqref{eq:factor} gives a useful identity.

\begin{proposition}[Cumulant moments]\label{prop:cumulants}
For every integer \(n\geq1\),
\begin{equation}\label{eq:cumulant-moment}
 \E R_X^n
 =
 \frac{2\kappa_{n+2}(X)}
 {(n+1)(n+2)V}.
\end{equation}
In particular,
\begin{equation}\label{eq:meanR}
 \E R_X=\frac{\kappa_3(X)}{3V}.
\end{equation}
\end{proposition}

\begin{proof}
Expanding both sides of \eqref{eq:factor} around zero gives
\[
 \sum_{k=2}^{\infty}\frac{\kappa_k(X)}{k!}t^k
 =
 \frac{Vt^2}{2}
 \sum_{n=0}^{\infty}\frac{\E R_X^n}{n!}t^n.
\]
Equating the coefficient of \(t^{n+2}\) proves \eqref{eq:cumulant-moment}.
\end{proof}

Since
\[
 \lim_{t\downarrow0}\Psi_X(t)=\E R_X,
\]
Theorem~\ref{thm:variational} implies the local obstruction
\begin{equation}\label{eq:local-lb}
 c_*(X)\geq\frac{\kappa_3(X)}{3V}.
\end{equation}

There is also a global obstruction. Let
\begin{equation}\label{eq:tau}
 \tau_X=\sup\{t>0:M_{R_X}(t)<\infty\}\in(0,\infty].
\end{equation}
By Theorem~\ref{thm:factor}, this is also the positive abscissa of convergence of \(K_X\).

\begin{proposition}[Boundary obstruction]\label{prop:boundary}
If \(\tau_X<\infty\), then
\begin{equation}\label{eq:boundary-lb}
 c_*(X)\geq\frac1{\tau_X}.
\end{equation}
Consequently,
\begin{equation}\label{eq:combined-lb}
 c_*(X)\geq
 \max\left\{
 \frac{\kappa_3(X)}{3V},
 \frac1{\tau_X}
 \right\},
\end{equation}
where \(1/\infty=0\).
\end{proposition}

\begin{proof}
If \(c<1/\tau_X\), then \(1/c>\tau_X\). An argument \(t\in(\tau_X,1/c)\) has \(M_{R_X}(t)=\infty\), so \eqref{eq:expdom} fails. This proves \eqref{eq:boundary-lb}; combine it with \eqref{eq:local-lb}.
\end{proof}

Neither lower bound in \eqref{eq:combined-lb} is universally exact. Section~\ref{sec:examples} gives cases where the optimum is controlled at the origin, at the boundary, or at an interior point.

\section{M-class optimality and equality laws}\label{sec:mclass}

\subsection{The local lower bound}

\begin{definition}[M-class]\label{def:Mclass}
Let \(R\) be nonnegative with mean \(m\in(0,\infty)\) and a positive exponential moment. Following \citet{KlarMuller2003}, \(R\) belongs to the M-class, written \(R\in\Mclass\), if
\begin{equation}\label{eq:Mclass}
 M_R(t)\leq\frac1{1-mt},
 \qquad 0\leq t<1/m.
\end{equation}
\end{definition}

\begin{theorem}[Local optimality]\label{thm:Mclass}
Suppose that \(X\) has a positive exponential moment and is not purely Gaussian. Then
\begin{equation}\label{eq:Mclass-equivalence}
 \boxed{
 c_*(X)=\frac{\kappa_3(X)}{3V}
 \quad\Longleftrightarrow\quad
 R_X\in\Mclass.
 }
\end{equation}
\end{theorem}

\begin{proof}
Put
\[
 m=\E R_X=\frac{\kappa_3(X)}{3V}.
\]
The non-Gaussian assumption and spectral positivity imply \(m>0\). Equation~\eqref{eq:local-lb} gives \(c_*(X)\geq m\). Equality holds if and only if \(m\) is feasible. By \eqref{eq:expdom}, feasibility of \(m\) is exactly \eqref{eq:Mclass}.
\end{proof}

Theorem~\ref{thm:Mclass} identifies when the local Taylor obstruction is globally sharp: this occurs precisely when \(R_X\) belongs to the M-class.

\subsection{MGF order versus coefficientwise bounds}

A familiar sufficient condition for \eqref{eq:feasible} is the coefficientwise Bernstein bound
\begin{equation}\label{eq:bernstein-cumulants}
 \kappa_k(X)
 \leq
 \frac{k!}{2}V c^{k-2},
 \qquad k\geq3.
\end{equation}
For the present class all cumulants of order at least three are nonnegative. Proposition~\ref{prop:cumulants} shows that \eqref{eq:bernstein-cumulants} is equivalent to
\begin{equation}\label{eq:moment-order}
 \E R_X^n\leq n!c^n,
 \qquad n\geq1.
\end{equation}
Thus the coefficientwise argument requires the moments of \(R_X\) to be bounded by the corresponding moments of an exponential variable of mean \(c\). The exact condition \eqref{eq:expdom} asks only for the MGF inequality on its relevant interval.

Moment order is strictly stronger than MGF order for general nonnegative laws; \citet{KlarMuller2003} give an explicit finite-support example. Theorem~\ref{thm:variational} works with the MGF order itself and therefore avoids the stronger coefficientwise assumption.

\subsection{Exact equality in the standard envelope}

The same factorization determines every law that attains equality in \eqref{eq:feasible}.

\begin{theorem}[Equality extremizers]\label{thm:extremizer}
Suppose that, for some \(c\geq0\) and \(\varepsilon>0\),
\begin{equation}\label{eq:envelope-equality}
 K_X(t)=\frac{Vt^2}{2(1-ct)}
 \qquad\text{for }0<t<\varepsilon.
\end{equation}
Then the following alternatives hold.
\begin{enumerate}
\item If \(c=0\), then \(X\sim N(0,V)\).
\item If \(c>0\), then \(X\) is a centered compound Poisson variable with Lévy measure
\begin{equation}\label{eq:extreme-levy}
 \nu(dx)=\frac{V}{2c^3}e^{-x/c}\,dx,
 \qquad x>0.
\end{equation}
Equivalently,
\begin{equation}\label{eq:extreme-cp}
 X\stackrel d=
 \sum_{i=1}^{N}J_i-\E\sum_{i=1}^{N}J_i,
 \quad
 N\sim\operatorname{Poisson}\left(\frac{V}{2c^2}\right),
 \quad
 J_i\sim\Exp(\text{mean }c).
\end{equation}
\end{enumerate}
Conversely, the Gaussian law in the first case and the compound Poisson laws in the second case satisfy \eqref{eq:envelope-equality} throughout their positive MGF domains.
\end{theorem}

\begin{proof}
Equations~\eqref{eq:factor} and \eqref{eq:envelope-equality} imply
\begin{equation}\label{eq:Rexp-mgf}
 M_{R_X}(t)=\frac1{1-ct}
 \qquad (0<t<\varepsilon).
\end{equation}

If \(c=0\), then \(M_{R_X}(t)=1\) near zero. Since \(R_X\geq0\), this forces \(R_X=0\) almost surely. The variable \(B\) is strictly positive almost surely, so \(Y_X=0\) almost surely. Hence \(H=V\delta_0\), which means \(a=V\) and \(\nu=0\). The centered law is \(N(0,V)\).

Let \(c>0\). MGF uniqueness in a neighborhood of zero gives
\[
 R_X\sim\Exp(\text{mean }c).
\]
In particular, \(R_X\) has no atom at zero. By \eqref{eq:Ratom}, \(a=0\). Its density is
\[
 q_X(r)=\frac1c e^{-r/c},
 \qquad r>0.
\]
Define, for \(r>0\),
\[
 Q(r)=\int_{(r,\infty)}(x-r)\,\nu(dx).
\]
Equation~\eqref{eq:Rdensity} gives
\begin{equation}\label{eq:Qexp}
 Q(r)=\frac{V}{2c}e^{-r/c}.
\end{equation}
We now recover \(\nu\) without assuming it has a density. Let \(\varphi\in C_c^\infty(0,\infty)\). Fubini's theorem and two integrations by parts give
\begin{align*}
 \int_0^\infty Q(r)\varphi''(r)\,dr
 &=
 \int_{(0,\infty)}
 \left\{\int_0^x(x-r)\varphi''(r)\,dr\right\}\nu(dx)\\
 &=
 \int_{(0,\infty)}\varphi(x)\,\nu(dx).
\end{align*}
The boundary terms at zero vanish because \(\varphi\) is compactly supported inside \((0,\infty)\). Thus \(Q''=\nu\) in the sense of distributions on \((0,\infty)\). Differentiating the smooth right-hand side of \eqref{eq:Qexp} twice yields
\[
 \nu(dx)=\frac{V}{2c^3}e^{-x/c}\,dx,
\]
which proves \eqref{eq:extreme-levy}.

The total mass of this Lévy measure is
\[
 \lambda=\nu((0,\infty))=\frac{V}{2c^2}.
\]
After normalization, its jump density is \(c^{-1}e^{-x/c}\), proving \eqref{eq:extreme-cp}. The centering term is determined by \(\E X=0\).

Conversely, for \eqref{eq:extreme-cp},
\begin{align*}
 K_X(t)
 &=
 \frac{V}{2c^2}
 \left\{
 \frac1{1-ct}-1-ct
 \right\}\\
 &=
 \frac{Vt^2}{2(1-ct)}.
\end{align*}
The Gaussian case is immediate.
\end{proof}

\begin{remark}\label{rem:extremizer}
Because equality on any open interval adjoining zero determines the MGF, the theorem characterizes the underlying law.
\end{remark}

\section{Structural operations}\label{sec:structure}

Convolution and scaling act directly on the normalized canonical measure, leading to the following properties of \(c_*\).

\begin{proposition}[Independent sums]\label{prop:sums}
Let \(X_1\) and \(X_2\) be independent centered spectrally positive infinitely divisible variables with variances \(V_1,V_2\in(0,\infty)\). Then
\begin{equation}\label{eq:Rmixture}
 \Law(R_{X_1+X_2})
 =
 \frac{V_1}{V_1+V_2}\Law(R_{X_1})
 +
 \frac{V_2}{V_1+V_2}\Law(R_{X_2}).
\end{equation}
Consequently,
\begin{equation}\label{eq:c-sum}
 c_*(X_1+X_2)
 \leq
 \max\{c_*(X_1),c_*(X_2)\}.
\end{equation}
\end{proposition}

\begin{proof}
The canonical measures add under independent convolution, and their total masses are the variances. Normalizing the sum gives the mixture in \eqref{eq:Rmixture}. Let \(c=\max\{c_*(X_1),c_*(X_2)\}\), first assuming both scales are finite. For \(0\leq t<1/c\), the MGF of the mixture is bounded by the same mixture of two copies of \((1-ct)^{-1}\), hence by \((1-ct)^{-1}\). The result follows from \eqref{eq:expdom}. The assertion is trivial if one of the scales is infinite.
\end{proof}

The inequality in \eqref{eq:c-sum} need not be an equality. In particular, adding an independent Gaussian component puts additional mass at zero in the transformed law and may reduce the optimal scale.

\begin{proposition}[Scaling and Lévy time]\label{prop:scale-time}
For \(b>0\),
\begin{equation}\label{eq:scaling}
 c_*(bX)=b\,c_*(X).
\end{equation}
Let \((L_t)_{t\geq0}\) be a spectrally positive Lévy process with finite nonzero variance at time one, and set \(X_t=L_t-\E L_t\). Then, for every \(t>0\),
\begin{equation}\label{eq:levy-time}
 c_*(X_t)=c_*(X_1).
\end{equation}
\end{proposition}

\begin{proof}
Under multiplication by \(b\), the canonical construction gives
\[
 R_{bX}\stackrel d=bR_X.
\]
Substitution in \eqref{eq:variational}, followed by the change of variable \(s=bt\), proves \eqref{eq:scaling}. For the Lévy process, the canonical measure and the variance at time \(t\) are both multiplied by \(t\). Their normalized ratio is unchanged, so
\[
 R_{X_t}\stackrel d=R_{X_1}.
\]
Equation~\eqref{eq:variational} proves \eqref{eq:levy-time}.
\end{proof}

Scaling \(X\) by \(b\) scales \(c_*\) by \(b\), whereas Lévy-process time leaves \(c_*\) unchanged because the canonical measure and the variance grow by the same factor.

\section{Compound Poisson laws and equilibrium distributions}\label{sec:cp}

Let
\begin{equation}\label{eq:CP}
 X=\sum_{i=1}^{N}J_i-\lambda\E J,
 \qquad
 N\sim\operatorname{Poisson}(\lambda),
\end{equation}
where \(J_1,J_2,\ldots\) are independent copies of a nonnegative variable \(J\), independent of \(N\), and
\[
 0<\E J^2<\infty.
\]
Then \(a=0\), \(\nu=\lambda\Law(J)\), and
\begin{equation}\label{eq:CPvariance}
 V=\lambda\E J^2.
\end{equation}
Proposition~\ref{prop:Rlaw} becomes
\begin{align}
 q_X(r)
 &=
 \frac{2\E(J-r)_+}{\E J^2},
 \label{eq:CPdensity}\\
 \Pp(R_X>r)
 &=
 \frac{\E(J-r)_+^2}{\E J^2}.
 \label{eq:CPsurvival}
\end{align}

For a nonnegative \(J\) with finite second moment, its second-order equilibrium distribution, denoted by \(J^{e2}\), is the law with density
\begin{equation}\label{eq:e2}
 f_{J^{e2}}(r)
 =
 \frac{2\E(J-r)_+}{\E J^2},
 \qquad r\geq0.
\end{equation}
This is one of the standard equivalent definitions of higher-order equilibrium laws \citep{NandaJainSingh1996,WillmotDrekicCai2005,NairPreeth2009}.

\begin{proposition}[Equilibrium representation]\label{prop:e2}
For the centered compound Poisson law \eqref{eq:CP},
\begin{equation}\label{eq:R-e2}
 R_X\stackrel d=J^{e2}.
\end{equation}
Consequently, \(c_*(X)\) depends on the jump law but not on the intensity \(\lambda\).
\end{proposition}

\begin{proof}
Equation~\eqref{eq:R-e2} follows by comparing \eqref{eq:CPdensity} with \eqref{eq:e2}. The variational formula then contains no \(\lambda\).
\end{proof}

If \(J\) has a positive exponential moment, then
\begin{equation}\label{eq:e2mean}
 \E J^{e2}=\frac{\E J^3}{3\E J^2}.
\end{equation}
Theorem~\ref{thm:Mclass} gives the following exact criterion.

\begin{corollary}[Compound Poisson M-class criterion]\label{cor:CPM}
Under a positive exponential moment,
\begin{equation}\label{eq:CPM}
 c_*(X)=\frac{\E J^3}{3\E J^2}
 \quad\Longleftrightarrow\quad
 J^{e2}\in\Mclass.
\end{equation}
\end{corollary}

There is also a purely equilibrium-order formulation. For nonnegative \(U,W\), write \(U\leq_{\mgf}W\) when \(M_U(t)\leq M_W(t)\) for every positive argument for which the comparison is defined. \citet[Theorem~3.5]{KlarMuller2003} show that a lifetime law \(F\) is in the M-class if and only if it dominates its first equilibrium law in MGF order. Since the first equilibrium transform of \(J^{e2}\) is the third-order equilibrium law \(J^{e3}\),
\begin{equation}\label{eq:e-order}
 J^{e2}\in\Mclass
 \quad\Longleftrightarrow\quad
 J^{e2}\geq_{\mgf}J^{e3}.
\end{equation}
Equations~\eqref{eq:CPM} and \eqref{eq:e-order} connect an ageing comparison directly to the best variance-exact concentration scale of the compound sum.

\section{Exact examples and three optimization mechanisms}\label{sec:examples}

\subsection{Gamma jumps}

Let \(J\sim\operatorname{Gamma}(\alpha,\text{scale } \beta)\), where \(\alpha,\beta>0\). The jump MGF is
\[
 M_J(t)=(1-\beta t)^{-\alpha},
 \qquad t<1/\beta.
\]
Using either \eqref{eq:factor} or the compound Poisson cumulant generating function gives, with \(x=\beta t\),
\begin{equation}\label{eq:gamma-MR}
 M_{R_X}(t)
 =
 \frac{2\{(1-x)^{-\alpha}-1-\alpha x\}}
 {\alpha(\alpha+1)x^2}.
\end{equation}

\begin{proposition}[Exact Gamma-jump scale]\label{prop:gamma}
For compound Poisson Gamma jumps,
\begin{equation}\label{eq:gamma-scale}
 \boxed{
 c_*(X)
 =
 \beta\max\left\{1,\frac{\alpha+2}{3}\right\}.
 }
\end{equation}
\end{proposition}

\begin{proof}
For \(|x|<1\), expand \eqref{eq:gamma-MR} as
\begin{equation}\label{eq:gamma-series}
 M_{R_X}(t)
 =
 \sum_{m=0}^{\infty}a_mx^m,
 \qquad
 a_m=\frac{2(\alpha+2)_m}{(m+2)!}.
\end{equation}
Here \(a_0=1\), and
\begin{equation}\label{eq:gamma-ratio}
 \frac{a_{m+1}}{a_m}
 =
 \frac{\alpha+m+2}{m+3}.
\end{equation}

Suppose first that \(0<\alpha\leq1\). Equation~\eqref{eq:gamma-ratio} is at most one, so \(a_m\leq1\) for every \(m\). Therefore
\[
 M_{R_X}(t)
 \leq
 \sum_{m=0}^{\infty}x^m
 =
 \frac1{1-\beta t},
 \qquad 0\leq t<1/\beta.
\]
Thus \(c=\beta\) is feasible. On the other hand, the MGF boundary is \(\tau_X=1/\beta\), so Proposition~\ref{prop:boundary} gives \(c_*(X)\geq\beta\).

Now suppose that \(\alpha\geq1\) and put
\[
 d=\frac{\alpha+2}{3}.
\]
For every \(m\geq0\),
\[
 \frac{\alpha+m+2}{m+3}\leq d;
\]
after clearing denominators this inequality is equivalent to \(m\leq\alpha m\). It follows that \(a_m\leq d^m\) and hence
\[
 M_{R_X}(t)
 \leq
 \frac1{1-d\beta t},
 \qquad 0\leq t<1/(d\beta).
\]
Therefore \(c=\beta d\) is feasible. The local lower bound is
\[
 \frac{\E J^3}{3\E J^2}
 =
 \frac{\beta(\alpha+2)}{3}
 =
 \beta d,
\]
so equality follows.
\end{proof}

For \(0<\alpha<1\), the optimum is controlled by the positive MGF boundary and is strictly larger than the local cumulant bound. For \(\alpha>1\), it is controlled at the origin and the transformed law belongs to the M-class. The two mechanisms coincide at \(\alpha=1\), the exponential-jump equality case of Theorem~\ref{thm:extremizer}.

\subsection{A bounded two-point jump law}

Boundary and local control do not cover every possibility. Consider
\begin{equation}\label{eq:two-point}
 \Pp(J=1)=\frac{39}{40},
 \qquad
 \Pp(J=5)=\frac1{40}.
\end{equation}
Its first four moments are
\begin{equation}\label{eq:two-moments}
 \E J=\frac{11}{10},
 \quad
 \E J^2=\frac85,
 \quad
 \E J^3=\frac{41}{10},
 \quad
 \E J^4=\frac{83}{5}.
\end{equation}
For \(R_X=J^{e2}\), equations~\eqref{eq:cumulant-moment} or direct integration give
\begin{equation}\label{eq:R-two-moments}
 m:=\E R_X=\frac{41}{48},
 \qquad
 \E R_X^2=\frac{83}{48}.
\end{equation}
If \(R_X\) belonged to the M-class, comparison of the second derivatives at zero in \eqref{eq:Mclass} would require
\[
 \E R_X^2\leq2m^2.
\]
But
\[
 \frac{83}{48}
 >
 2\left(\frac{41}{48}\right)^2,
\]
so \(R_X\notin\Mclass\). Theorem~\ref{thm:Mclass} therefore gives
\begin{equation}\label{eq:two-strict}
 c_*(X)>\frac{41}{48}.
\end{equation}

Since \(J\) is bounded, \(M_{R_X}\) is entire and \(\tau_X=\infty\); the boundary obstruction is zero. More explicitly,
\begin{equation}\label{eq:two-MR}
 M_{R_X}(t)
 =
 \frac{
 2\left\{
 (39e^t+e^{5t})/40-1-11t/10
 \right\}}
 {(8/5)t^2}.
\end{equation}
The function \(\Psi_X\) is continuous on \((0,\infty)\),
\[
 \lim_{t\downarrow0}\Psi_X(t)=\frac{41}{48},
 \qquad
 \lim_{t\to\infty}\Psi_X(t)=0.
\]
Together with \eqref{eq:two-strict}, these limits imply that the supremum is attained at some interior point.

Numerical maximization of the explicit expression \eqref{eq:two-MR} gives
\begin{equation}\label{eq:two-numeric}
 t_*\approx0.473723112,
 \qquad
 c_*(X)\approx0.895416096.
\end{equation}
At an interior stationary point, the equation used for the computation is
\begin{equation}\label{eq:stationary}
 tM_{R_X}'(t)
 =
 M_{R_X}(t)\{M_{R_X}(t)-1\}.
\end{equation}
The computation establishes existence, but not uniqueness, of an interior maximizer.

\subsection{Interpretation}

The examples distinguish three ways in which the variational supremum can be determined:
\begin{enumerate}
\item \emph{Local control}: for a non-Gaussian law, the supremum in \eqref{eq:variational} is determined in the limit \(t\downarrow0\), equivalently \(R_X\in\Mclass\).
\item \emph{Boundary control}: the positive MGF abscissa forces the optimal pole, as for Gamma jumps with \(0<\alpha<1\).
\item \emph{Interior control}: \(\Psi_X\) exceeds both endpoint obstructions and attains its maximum at a positive finite argument, as in \eqref{eq:two-point}.
\end{enumerate}
The third case explains why the exact variational formula cannot in general be replaced by the maximum of the local and boundary lower bounds.

\section{Discussion}\label{sec:discussion}

The transform \(R_X\) organizes the two tails differently. Its nonnegativity gives the variance-exact sub-Gaussian estimate at negative arguments, while its positive MGF determines the optimal sub-gamma scale. Near zero it records the normalized third cumulant. In the compound Poisson case it is the second-order equilibrium distribution of the jump law, so the concentration problem becomes a reliability-theoretic comparison.

Spectral positivity is essential to this reduction because it makes \(R_X\) nonnegative and permits comparison with an exponential law. Signed jump measures lose that order structure. The same obstacle appears in higher dimensions, where neither a scalar equilibrium transform nor a single scale parameter captures all directions.

\section*{Statements and Declarations}
\addcontentsline{toc}{section}{Statements and Declarations}

\bmhead{Funding}

The author received no specific funding for this work.

\bmhead{Competing interests}

The author declares no competing interests.

\bmhead{Data availability}

No datasets were generated or analyzed during the current study.

\bmhead{Author contributions}

The sole author developed the theory, verified the calculations, and wrote the manuscript.

\bmhead{Ethics approval}

Not applicable.

\bmhead{Use of generative artificial intelligence}

OpenAI Codex assisted with English translation and Lean formalization of the proofs.

\bibliography{references}

\end{document}